\documentclass[11pt]{amsart}

\def\NZQ{\mathbb}               
\def\NN{{\NZQ N}}

\def\CC{{\NZQ C}}

\def\CC{{\mathcal C}}

\let\union=\cup

\let\ol=\overline
\let\wt=\widetilde

\def\opn#1#2{\def#1{\operatorname{#2}}} 
\opn\chara{char} \opn\length{\ell} \opn\pd{pd} \opn\rk{rk}
\opn\projdim{proj\,dim} \opn\injdim{inj\,dim} \opn\rank{rank}
\opn\depth{depth} \opn\codepth{codepth} \opn\grade{grade}
\opn\height{height} \opn\embdim{emb\,dim} \opn\codim{codim}

\opn\Tr{Tr} \opn\bigrank{big\,rank}
\opn\superheight{superheight}\opn\lcm{lcm}
\opn\trdeg{tr\,deg}%
\opn\reg{reg} \opn\lreg{lreg} \opn\skel{skel} \opn\Gr{Gr}
\opn\dim{dim} \opn\arithdeg{arithdeg}
\opn\dev{dev}

\opn\Gin{Gin}

\def\CC{{\mathcal C}}

\def\MM{{\mathcal M}}

\opn\inii{in} \opn\inim{inm} \opn\rate{rate}

\newtheorem{theorem}{Theorem}[section]
\newtheorem{lemma}[theorem]{Lemma}
\newtheorem{corollary}[theorem]{Corollary}
\newtheorem{proposition}[theorem]{Proposition}
\newtheorem{remark}[theorem]{Remark}

\newtheorem{example}[theorem]{Example}

\newtheorem{definition}[theorem]{Definition}

\begin{document}
\markboth{Giancarlo Rinaldo}
{Cohen--Macaulay binomial edge ideals}

\title{COHEN--MACAULAY BINOMIAL EDGE IDEALS OF CACTUS GRAPHS
}


\author{Giancarlo Rinaldo}
\address{Department of Mathematics\\
University of Trento\\
via Sommarive, 14\\
38123 Povo (Trento), Italy
}

\maketitle

\begin{abstract}
We classify the Cohen-Macaulay binomial edge ideals of cactus and bicyclic graphs.
\end{abstract}


\section*{Introduction}
In 2010, binomial edge ideals were introduced in \cite{HH} and appeared independently also in \cite{MO}. Let $S = K[x_1,\ldots, x_n, y_1,\ldots, y_n]$ be the polynomial ring in $2n$ variables with coefficients in a field $K$. Let $G$ be a graph on vertex set $[n]$ and edges $E(G)$. For each  $\{i,j\}\in E(G)$ with $i < j$, we associate a binomial $f_{ij} = x_iy_j - x_jy_i$. The ideal $J(G)$ of $S$ generated by $f_{ij} = x_iy_j - x_jy_i$ such that $i<j$ , is called \textit{the binomial edge ideal} of $G$. Any ideal generated  by a set of $2$-minors  of a $2\times n$-matrix of indeterminates may be viewed as the binomial edge ideal of a graph. 


Algebraic properties of binomial edge ideals in terms of properties of the underlying graph were studied in \cite{HH}, \cite{HEH} and \cite{RR}. In \cite{HEH} and \cite{RR} the authors considered the Cohen-Macaulay property of these graphs. Recently nice results on Cohen-Macaulay bipartite graphs and blocks have been obtained (see \cite{KM}, \cite{BMS} and \cite{BN}). 

However, the classification of Cohen-Macaulay binomial edge ideals in terms of the underlying graphs is still widely open and, as in the case of monomial edge ideals introduced in \cite{Vi2}, it seems rather hopeless to give a full classification.

The aim of this paper is to extend the results of \cite{Ri} where we classify Cohen-Macaulay and unmixed binomial edge ideals $J(G)$ with deviation, namely the difference between the minimum number of the generators and the height of $J(G)$, less than or equal to $1$. This invariant has an interesting combinatorial interpretation: if $J(G)$ is unmixed than its deviation is $|E(G)|-n + c$ where $c$ is the number of connected components of $G$ (see Remark 2.1 of \cite{Ri}). Hence deviation represents the minimum number of edges that must be removed from the graph to break all its cycles making it into a forest (see Chapter 4 of \cite{Ha}).

In section \ref{sec:Cactus} we give a classification Cohen-Macaulay and unmixed binomial edge ideals $J(G)$ when $G$ is a cactus graph, a graph whose blocks are cycles. This is  a natural extension to the result obtained in \cite{Ri} and useful to study binomial edge ideals by the deviation invariant. 

In section \ref{sec:Bicycle} as an application of the results obtained in section \ref{sec:Cactus} we classify the Cohen-Macaulay and unmixed binomial edge ideals $J(G)$ when $G$ is a bicyclic graph, that is the case of deviation $2$.

\section{Preliminaries}\label{sec:pre}
In this section we recall some concepts and notation on graphs and on simplicial
complexes that we will use in the article.

Let $G$ be a simple graph with vertex set $V(G)$ and edge set $E(G)$. A subset $C$ of $V(G)$ is called a \textit{clique} of $G$ if for all $i$ and $j$ belonging to $C$ with $i \neq j$ one has $\{i, j\} \in E(G)$. A vertex of a graph is called a \textit{cutpoint} if the removal of the vertex increases the number of connected components. A connected subgraph of $G$ that has no cutpoint and is maximal with respect to this property is a \textit{block}. A block graph $B(G)$ is a graph whose vertices are the blocks of $G$ and two vertices are adjacent whenever the corresponding blocks contain a common cutpoint of $G$. A connected graph is a \textit{cactus} if its blocks are cycles or edges.

Set $V = \{x_1, \ldots, x_n\}$. A \textit{simplicial complex} $\Delta$ on the vertex set $V$ is a collection of subsets of $V$ such that: $(i)$ $\{x_i\} \in \Delta$  for all $x_i \in V$; $(ii)$ $F \in \Delta$ and $G\subseteq F$ imply $G \in \Delta$.
An element $F \in \Delta$ is called a \textit{face} of $\Delta$. A maximal face of $\Delta$  with respect to inclusion is called a \textit{facet} of $\Delta$.
A vertex $i$ of $\Delta$ is called a free vertex of $\Delta$ if $i$ belongs to exactly one facet.
The \textit{clique complex} $\Delta(G)$ of $G$ is the simplicial complex whose faces are the cliques of $G$. Hence a vertex $v$ of a graph $G$ is called \textit{free vertex} if it belongs to only one clique of $\Delta(G)$.

We need notation and results  from \cite{HH} (section 3) that we recall for the sake of completeness.
Let $T\subseteq [n]$, and let $\ol{T}=[n]\setminus T$. Let $G_1,\ldots,G_{c(T)}$ be the connected components of the induced subgraph on $\ol{T}$, namely $G_{\ol{T}}$. For each $G_i$, denote by $\wt{G}_i$ the complete graph on the vertex set $V(G_i)$. We set 
\begin{equation}\label{eq:prime}
P_T(G)=(\bigcup_{i\in T}\{x_i,y_i\},J(\wt{G}_1),\ldots,J(\wt{G}_{c(T)}) ), 
\end{equation}
$P_T(G)$ is a prime ideal. Then $J(G)$ is a radical ideal and 
\[
J(G)=\bigcap_{T\subset [n]}P_T(G)
\]
is its primary decomposition (see Corollary 2.2 and Theorem 3.2 of \cite{HH}). If there is no possible confusion, we write simply $P_T$ instead of $P_T(G).$ Moreover, $\height P_T=n+|T|-c(T)$ (see \cite[Lemma 3.1]{HH}).
 We denote by $\MM(G)$ the set of minimal prime ideals of $J(G)$.

If each $i\in T$ is a cutpoint of the graph $G_{\ol{T}\cup \{i\}}$, then we say that $T$ is a {\em cutset for $G$.} We denote by $\mathcal{C}(G)$ the set of all cutsets for $G$.
%
\begin{lemma}\label{lem:cutpoint}\cite{HH}
 $P_T(G)\in \MM(G)$ if and only if $T\in\mathcal{C}(G)$.
\end{lemma}

\begin{lemma}\label{lem:unmixwd}\cite{RR} Let $G$ be a connected graph. Then $J(G)$ is unmixed if and only if for all $T\in\mathcal{C}(G)$ we have $c(T)=|T|+1$.
\end{lemma}
The following observation gives motivation to consider the block graphs in this context.
\begin{proposition}\label{tree}
 Let $J(G)$ be unmixed. Then the block graph $B(G)$ is a tree.
\end{proposition}
\begin{proof}
By Theorem 3.5 of \cite{Ha} each block of the block graph is a complete graph. We recall that two vertices in $B(G)$ are adjacent when the corresponding blocks contain a common cutpoint of $G$. Moreover for each cutpoint $v$, $\{v\}\in \CC(G)$. By Lemma \ref{lem:unmixwd} the assertion follows.
\end{proof}
We observe the following (see also  Remark 3.1.(i) of \cite{BMS})
\begin{proposition}\label{Pro:cutset}
 Let $v$ be a vertex of $G$ with neighbor set $N(v) = \{u \in V(G)\mid \{u,v\}\in E(G)\}$ and let $G_{\ol{v}}$ the graph such that $V(G_{\ol{v}})=V(G)$ and  $E(G_{\ol{v}})=E(G)\cup \{\{u_1,u_2\} \mid u_1\neq u_2, u_1,u_2\in N(v)\}$.
 Then 
 \[
  \CC(G_{\ol{v}})=\{T\in \mathcal C(G)\mid v\not\in T\}.
 \]
\end{proposition}
\begin{proof}
If $v$ is a free vertex then $E(G)=E(G_{\ol{v}})$. In fact by definition there is only one clique containing $v$, that is all vertices adjacent to $v$ are adjacent to each other. Hence we assume $v$ is not a free vertex in $G$.  By definition of the graph $G_{\ol{v}}$ $v$ is a free vertex of $G_{\ol{v}}$. This implies by Proposition 1.1  of \cite{RR} that there are no cutsets of $G_{\ol{v}}$ containing $v$. 

Let $S\subset V(G)$ such that $v\notin S$, and let
 \[
 G_1,\ldots, G_{c_1}
\]
be the connected components of the graph $G_{\ol{S}}$, with $v\in V(G_1)$. Each $\{v,u\}\in E(G)$ either intersects $S$, hence is not in $E(G_{\ol{S}})$, or is an edge of $G_1$ since it is connected through $v$ in $G_1$. In the same way, let $H=G_{\ol{v}}$ and let
 \[
 H_1,\ldots, H_{c_2}
\]
be the connected components of the graph $H_{\ol{S}}$, with $v\in V(H_1)$. Each $\{v,u\}\in E(H)$ and each $\{u_1,u_2\}\in E(H)$ with $u_1$, $u_2\in N(v)$ either intersects $S$ or is an edge of $H_1$. This implies that $H_i=G_i$ for $i=2,\ldots,c_1$ and $c_1=c_2$. Moreover
\begin{equation}\label{eq:samevertices}
 V(G_i)=V(H_i),\hspace{1 cm}\forall i=1,\ldots,c_1.
\end{equation}
In particular this is true for all the cutsets. That is for all $T\in \CC(G)$ then  $i\in T$ is a cutpoint of the graph $G_{\ol{T}\cup \{i\}}$, and the number of connected components decreases when $S=T\setminus \{i\}$. The same happens for $H$, hence $T\in \CC(H)$. The same argument works in the other direction.
\end{proof}

\begin{corollary}\label{Lem:cutset}
 Let $v$ be a vertex of $G$ that is not free vertex and let $G_{\ol{v}}$ the graph such that $V(G_{\ol{v}})=V(G)$ and  $E(G_{\ol{v}})=E(G)\cup \{\{u_1,u_2\} \mid u_1\neq u_2, u_1,u_2\in N(v)\}$. Then 
\[
   J(G)=J(G_{\ol{v}})\cap Q_v
\]
with $Q_v = \bigcap_{T\in\CC(G), v\in T} P_T(G)$.
\end{corollary}
\begin{proof}
Let $v\in V(G)$ then $J(G)=Q_{\ol{v}}\cap Q_v$ with 
\[
Q_v = \bigcap_{T\in\CC(G), v\in T} P_T(G),\hspace{1cm}Q_{\ol{v}} = \bigcap_{T\in\CC(G), v\notin T} P_T(G). 
\]
Applying Proposition \ref{Pro:cutset}, we have that the cutsets in $\CC(G_{\ol{v}})$ are exactly the cutsets in $\CC(G)$ not containing $v$. Moreover the connected components induced by any $T$ in $G$ and in $G_{\ol{v}}$ have the same set of vertices as stated in the proof of Proposition \ref{Pro:cutset} (see   \eqref{eq:samevertices}). Using the notation introduced in the mentioned proof $\wt{H}_i=\wt{G}_i$  for each connected component, that is $$J(G_{\ol{v}})=Q_{\ol{v}}.$$
\end{proof}

\begin{example}\label{Ex:C4DK4} Sometime the ideal $Q_v$ in Corollary \ref{Lem:cutset} has a natural interpretation. For example let $G=C_4$ with vertices $V(C_4)=\{1,\ldots,4\}$ and edges $\{\{1,2\},\{2,3\},\{3,4\},\{1,4\}\}$. Then we obtain
\begin{equation}\label{eq:C4toD}
 J(G)=J(G_{\bar{1}})\cap Q_1=J(D)\cap (x_1,y_1, x_3, y_3)
\end{equation}
where $D$, known as diamond graph, has $V(D)=V(C_4)$ and $E(D)=E(C_4)\cup \{\{2,4\}\}$ and the ring $S/Q_v$ is related to the isolated vertex $2$ and $4$. By similar argument we obtain  
\begin{equation}\label{eq:DtoK4}
 J(D)=J(D_{\bar{2}})\cap Q_2=J(K_4)\cap (x_2,y_2, x_4, y_4)
\end{equation}
where $K_4$ is the complete graph with $4$ vertices. If $G=C_5$ such an interpretation fails.  Let $V(C_5)=\{1,\ldots,5\}$ and edges  $\{\{1,2\},\{2,3\},\{3,4\},\{4,5\},\{1,5\}\}$. If we consider $J(G)=J(G_{\bar{1}})\cap Q_1$, then $G_{\bar{1}}$ is the graph obtained adding the edge $\{2,5\}$ to the cycle $C_5$, as expected by Corollary \ref{Lem:cutset}, but 
\[
 Q_1=(x_1,y_1, f_{2,3}, f_{4,5}, x_3 x_4, x_3 y_4, y_3 x_4, y_3 y_4).
\]
\end{example}
Passing from a block that is not complete to a complete one (see \eqref{eq:C4toD},\eqref{eq:DtoK4} in Example \ref{Ex:C4DK4} and Figure \ref{CM-C4toK4}), is useful for our aim thanks to the following nice result (see Theorem 1.1 of \cite{HEH}) that we state using our notation  
\begin{theorem}\label{Th:EHH}
Let $G$ be a graph whose blocks are complete graphs. The following conditions are equivalent:
\begin{enumerate}
 \item $J(G)$ is Cohen-Macaulay; 
 \item $J(G)$ is unmixed; 
 \item $B(G)$ is a tree. 
\end{enumerate}
\end{theorem}

\section{Classification of Cohen-Macaulay cactus graphs}\label{sec:Cactus}
In this section we provide a classification of Cohen-Macaulay binomial edge ideal $J(G)$ when $G$ is a cactus graphs. 
Since a binomial edge ideal $J(G)$ is Cohen--Macaulay (resp. unmixed) if and only
if $J(H)$ is Cohen--Macaulay (resp. unmixed) for each connected component $H$ of $G$,
we assume from now on that the graph G is connected unless otherwise stated.

We start by the following
\begin{proposition}\label{the:C3AndC4Only}
 Let $J(G)$ be an unmixed binomial edge ideal. If a cycle $C_l$ is a block of $G$ then $l\in \{3,4\}$.
\end{proposition}

\begin{proof}
We begin observing by \cite{HEH} and \cite{Ri} that unmixed binomial edge ideals containing blocks that are cycles of length $3$ and $4$ exist. We  assume that exists a cycle $C_l$ in $G$ with $l\geq 5$ and $J(G)$ unmixed. Since $C_l$ is a block we represent $G$ as
\begin{equation}\label{eq:cactus}
  G=C_l\cup \left(\bigcup_{i=1}^r G_i\right) 
\end{equation}
where $r\geq 0$, $G_i$ are subgraphs of $G$ for all $i$,  $|V(G_i)\cap V(C_l)|=1$. Since $C_l$ is a block we have that $\{v\}=V(G_i)\cap V(C_l)$ is a cutpoint and all the paths between $u\in V(G_i)$ and $w\in V(G_j)$  with $i\neq j$ pass through $v$ (see Theorem 3.1 of \cite{Ha}).

Let $V(C_l)=\{i_1,i_2,\ldots,i_l\}$ such that $\{i_j,i_{j+1}\}$ is an edge of $G$ with $j=1,\ldots,l-1$ and $\{i_1,i_l\}$ is an edge of $G$, too.  We observe that if $T$ is $\{i_j,i_k\}$  such that $\{i_j,i_k\}\notin E(C_l)$  then $T\in \CC(G)$. Moreover $G_{\ol{T}}$ has at least two connected components induced by the remaining $l-2$ vertices of $C_l$. Assuming $G$ is unmixed there are exactly $3$ connected components in $G_{\ol{T}}$ and this implies that one of the vertices in $T$ is a cutpoint and the other is not a cutpoint. Thanks to this observation we easily obtain a contradiction. We give the proof for the sake of completeness. We observe that in the cutset $\{i_1,i_3\}$ we may assume without loss of generality that $i_1$ is the cutpoint and $i_3$ is not a cutpoint. By the same argument focusing on $\{i_1,i_{l-1}\}$, $i_{l-1}$ is not a cutpoint. If we consider $\{i_2,i_{l-1}\}$ and $\{i_3,i_{l}\}$ we obtain that $i_2$ and $i_l$ are cutpoints.  Let $T_1=\{i_2,i_l\}$ since $G_{\ol{T}_1}$ has $4$ connected components we obtain a contradiction. 
 \end{proof}

\begin{proposition}\label{Pro:C4}
 Let $J(G)$ be an unmixed binomial edge ideal. If a cycle $C_4$ is a block of $G$ then it satisfies the following $(C4)$-condition
\begin{itemize}
 \item there are exactly two cutpoints in $C_4$ and they are adjacent.
\end{itemize}
\end{proposition}
\begin{proof}
By Proposition 2 of \cite{Ri} we know that a graph satisfying the thesis exists. We prove that all the other cases are not unmixed.
Let $V(C_4)=\{1,2,3,4\}$ and let $E(C_4)=\{\{1,2\},\{2,3\},\{3,4\},\{1,4\}\}$.  We assume the same representation of $G$ given in \eqref{eq:cactus}.
 Suppose that either $r=0$ or $r=1$ with cutpoint $1$. Then $c(\{2,4\})=2$. Hence it is not unmixed. Let $r \geq 2$ and assume that $1$ and $3$ are cutpoints then $c(\{1,3\})=4$. Hence $J(G)$ is not unmixed.
\end{proof}
\begin{definition}\label{def:indecomposable}
  A graph $G$ is decomposable (resp. indecomposable)  if exists (resp. does not exist) a decomposition
  \begin{equation}\label{eq:dec2}
  G=G_1\cup G_2 
  \end{equation}
  with $\{v\}=V(G_1)\cap V(G_2)$ such that $v$ is a free vertex of $\Delta(G_1)$ and $\Delta(G_2)$. By a recursive decomposition \eqref{eq:dec2} applied to each $G_1$ and $G_2$, after a finite number of steps we obtain
  \begin{equation}\label{eq:decr}
     G=G_1\cup \cdots \cup G_r
  \end{equation}
  where $G_1,\ldots,G_r$ are indecomposable and for $1\leq i<j\leq r$ either $V(G_i)\cap V(G_j) = \emptyset$ or $V(G_i)\cap V(G_j) =\{v\}$ and  $v$ is a free vertex of $\Delta(G_i)$ and $\Delta(G_j)$. The decomposition \eqref{eq:decr} is unique up to ordering and we say that $G$ is decomposable into indecomposable graphs $G_1,\ldots,G_r$.
\end{definition}

\begin{lemma}[See \cite{RR}]\label{The:free}
Let $G$ be a decomposable graph with decomposition $G=G_1\cup G_2$ and $V(G_1)\cap V(G_2)=\{v\}$. Then $J(G)$ is Cohen-Macaulay (resp. unmixed) if and only if $J(G_1)$ and $J(G_2)$ are Cohen-Macaulay (resp. unmixed).
\end{lemma}

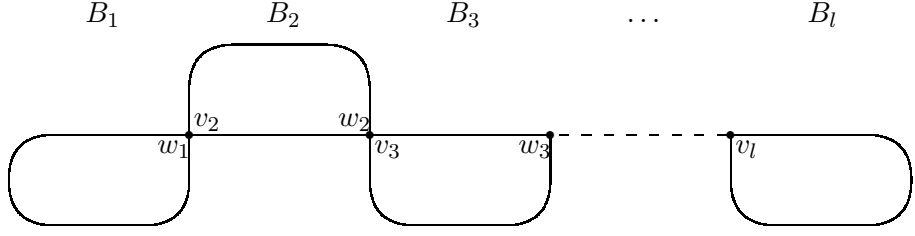
\begin{figure}
\begin{center}
\setlength{\unitlength}{.60cm}
\begin{picture}(20,5)

\put(01,02){\line(1,0){11}}
\multiput(11.2,02)(0.5,0){10}{\line(1,0){.2}}
\put(16,02){\line(1,0){3}}

\qbezier(00,01)(0,02)(01,02)
\qbezier(00,01)(0,00)(01,00)
\put(01,00){\line(1,0){2}}
\qbezier(03,00)(04,00)(04,01)

\put(04,01){\line(0,1){2}}
\qbezier(04,03)(04,04)(05,04)
\put(05,04){\line(1,0){2}}
\qbezier(07,04)(08,04)(08,03)
\put(08,03){\line(0,-1){2}}

\qbezier(08,01)(08,00)(09,00)
\put(09,00){\line(1,0){2}}
\qbezier(11,00)(12,00)(12,01)
\put(12,01){\line(0,1){1}}

\put(16,02){\line(0,-1){1}}
\qbezier(16,01)(16,00)(17,00)
\put(17,00){\line(1,0){2}}
\qbezier(19,00)(20,00)(20,01)
\qbezier(20,01)(20,02)(19,02)

\put(01.7,4.5){$B_1$}
\put(04,02){\circle*{.2}}
\put(03.3,01.6){$w_1$}

\put(05.7,4.5){$B_2$}
\put(08,02){\circle*{.2}}
\put(04.1,02.2){$v_2$}
\put(07.3,02.2){$w_2$}

\put(09.7,4.5){$B_3$}
\put(12,02){\circle*{.2}}
\put(08.1,01.6){$v_3$}
\put(11.3,01.6){$w_3$}

\put(13.7,4.5){$\ldots$} 

\put(17.7,4.5){$B_l$}
\put(16,02){\circle*{.2}}
\put(16.1,01.6){$v_l$}

\end{picture}
 
\end{center}
\caption{$B(G)$ is a path.}\label{BPath} 
\end{figure}
Our aim is to prove that every Cohen-Macaulay cactus graph $G$ is decomposable into indecomposable graphs $G_1,\ldots,G_r$ such that the block graph $B(G_i)$ is a path for $1\leq i\leq r$. Hence it is necessary as a first step to classify the indecomposable Cohen-Macaulay cactus graphs whose block graph is a path. To reach this goal we use from now on the following notation. Let $G$ be a graph such that $B(G)$ is a path defined on the following sets (see Figure \ref{BPath})
\begin{equation}\label{pathnotation}
\begin{array}{ll}
 \mbox{vertices of }  B(G) & \{ B_1,\ldots, B_l \} \\
 \mbox{edges of }     B(G) & \{ \{B_i,B_{i+1}\} \mid i=1,\ldots,l-1\} \\
 \mbox{cutpoints of }  G  & \{w_i=v_{i+1} \mid i=1,\ldots,l-1\}  
 \end{array}
\end{equation}
and such that $w_i=v_{i+1}\in V(B_i)\cap V(B_{i+1})$, with $i=1,\ldots,l-1$.
\begin{lemma}\label{Lem:pathgraph} Let $G$ be a graph such that $B(G)$ is a path. We use notation \eqref{pathnotation}. If $|V(B_i)|\geq 3$ for all $1<i<l$ then the power set of $\{w_1,\ldots,w_{l-1}\}$ is a subset of $\CC(G)$. 
\end{lemma}
\begin{proof}
We use induction on $t$ the cardinality of $T\in\CC(G)$. If $t=1$ since each $w_i$ is a cutpoint of $G$ the claim  follows. Let $2\leq t < l-1$ and 
\[
 T=\{w_{i_1}, w_{i_2},\ldots w_{i_t}\} \mbox{ with }1\leq i_1<i_2<\cdots<i_t\leq l-1.
\]
Let $w_j\in \{w_1,\ldots,w_{l-1}\}\setminus T$ and assume without loss of generality that $j<i_1$ and $i_1\geq 2$. We observe that the graph $G_{\ol{T}}$ has the connected component $H=B_1\cup \ldots\cup B_{i_1-1}\cup B_{i_1}'$ where $B_{i_1}'$ is obtained removing the vertex $w_{i_1}$ from the block $B_{i_1}$. Since by hypothesis $B_{i_1}$ has more than two vertices, $B_{i_1}'$ contains at least one vertex, $v_{i_1}'$, that is not a cutpoint in $G$. Hence $H_{\ol{\{w_j\}}}$ has two connected components: one containing a vertex in $V(B_1)\setminus \{w_1\}$ and one containing the vertex $v_{i_1}'$. That is  $w_j$ is a cutpoint of $G_{\ol{T}}$ and $T\cup \{w_j\}\in \CC(G)$.
\end{proof}
By Proposition \ref{the:C3AndC4Only}, Example \ref{Ex:C4DK4} and Theorem \ref{Th:EHH} is useful for our aim to compute the primary decompositions of $J(G)$  whose blocks are complete graphs, cycles $C_4$ and diamond graphs $D$ (see also Figure \ref{CM-C4toK4}).

\begin{proposition}\label{lem:pathcutset}
Let $G$ be a graph such that $B(G)$ is a path.  We use notation \eqref{pathnotation}. Let $B_1=K_{m_1}$, $B_l=K_{m_l} $ with $m_1$, $m_l\geq 2$ and 
\[
 B_i\in \{C_4, D, K_{m_i} \mid m_i\geq 3 \}\mbox{ for }2\leq i\leq l-1
\]
with the following labelling on the vertices of the blocks $B_i\in \{C_4,D\}$:
\begin{itemize}
\item $V(C_4)=\{v_i,v_i',w_i,w_i'\}$, $E(C_4)=\{\{v_i,w_i\},\{v_i,v_i'\},\{v_i',w_i'\},\{w_i,w_i'\}\}$;
\item $V(D)=\{v_i,v_i',w_i,f_i\}$, $E(D)=\{\{v_i,w_i\},\{v_i,v_i'\},\{v_i',f_i\},\{w_i,f_i\},\{v_i',w_i\}\}$.
\end{itemize}
Then $T\in \CC(G)$ if and only if $T\subseteq V\sqcup V'$ where $V=\{w_1,\ldots,w_{l-1}\}$,  
\[
V'=\left(\bigcup_{B_i=C_4} \{v_i',w_i'\}\right) \cup \left(\bigcup _{B_i=D} \{v_i'\}\right)                                                                                                                 
\]
and satisfying the following conditions:
\begin{enumerate}
 \item if $v_i'\in T$ then $w_i\in T$ and  $v_i\notin T$; 
 \item if $w_i'\in T$ then $v_i\in T$ and  $w_i\notin T$. 
\end{enumerate}
\end{proposition}
\begin{proof} 
For the sake of completeness we give the equivalent conditions to (1) and (2) that are:
\begin{enumerate}
 \item[(1')] if $v_i\in T$ or $w_i\notin T$  then $v_i'\notin T$; 
 \item[(2')] if $w_i\in T$ or $v_i\notin T$  then $w_i'\notin T$. 
 \end{enumerate}
Suppose $T$ is a cutset. Then $v\in T$ is not a free vertex. Hence it follows $T\subseteq V\sqcup V'$. In fact a vertex $f_i$ in a diamond graph is a free vertex in the clique $\{v_i',w_i,f_i\}$ and the same holds for all the vertices that are not cutpoints in a block that is a complete graph. Moreover, if $T\subseteq V$ then it satisfies condition 1. and 2. in a trivial way. Hence we assume $v'\in T$ with $v'\in V'$. That is either $v'=v_i'\in V(B_i)$ with $B_i\in\{C_4, D\}$ or $v'=w_i'\in V(B_i)$ with $B_i=C_4$. We consider only the case $v'=v_i'\in V(B_i)$ with $B_i=C_4$ since the other cases follow by similar argument. Since $T$ is a cutset, the graph $G_{S}$ with $S=\ol{T}\cup\{v_i'\}$ has a cutpoint in $v_i'$. Hence the two vertices $v_i$ and $w_i'$ that are adjacent to $v_i'$ in $G$ must belong to $G_S$, that is $v_i$ and $w_i'$ are not in $T$. By the same reason since $v_i'$ is a cutpoint in $G_S$ and since $w_i$ is adjacent to $v_i$ and $w_i'$ in $G$, too,  then $w_i\notin S$ and  $w_i\in T$. That is condition (1) is satisfied.

Now suppose $T\subseteq V\sqcup V'$ satisfying conditions (1) and (2) and let $n'=|V'|$.  We prove that $T$ is a cutset by induction on $n'$. If $n'=0$ that is $T\subseteq V$, then it is a cutset by Lemma \ref{Lem:pathgraph}. 

Let $n'\geq 1$ and  let $v_2'\in V'$ then $B_2\in \{C_4,D\}$ (see Figure \ref{CM-C4toK4}). We assume $B_2=C_4$ since the other case is similar. By conditions (1) and (2'). $w_2\in T$ and neither $v_2$ nor $w_2'$  belongs to $T$. That is the graph $G_{\ol{T}}$ contains $2$ connected components, $B_1$ and the isolated vertex $w_2'$, and a subgraph of $H=B_3'\union B_4\union \cdots \union B_l$, where $B_3'$ is obtained removing the vertex $w_2=v_3$ from $B_3$.  

We claim that $T'=T\setminus\{v_2',w_2\}$ is a cutset of $H$. Moreover $H_{\ol{T'}}$ has a connected component containing a vertex in $B_3$ that is adjacent to $w_2$. By the claim it easily follows that $T$ is a cutset of $G$ and $G_{\ol{T}}$ contains the following connected components: $B_1$, the isolated vertex $w_2'$, and the connected components of $H_{\ol{T'}}$.

If $B_3$ is a complete graph  then $B_3'$ is a complete graph, too. Hence $H$ satisfies the hypothesis of our Proposition and $T'$ is a cutset by induction hypothesis.
If $B_3=D$ then $B_3'$ is the complete graph $K_3$. Moreover $v_3'\notin T$ by condition (1') applied to $w_2=v_3\in T$. Hence the set $T'$ contains at most the cutpoint $w_3$. Also in this case the graph $H$ satisfies the hypothesis of our Proposition and $T'$ is a cutset by induction hypothesis. If $B_3=C_4$ then $B_3'$ is a path of length $2$ whose edges are 
\[
 \{v_3',w_3'\},\{w_3',w_3\}.
\]
We note that also in this case $v_3'\notin T$ by condition (1'), and $T$ contains at most one of the vertices $w_3$ and $w_3'$. In fact $v_3'$ is the free vertex of the path and only one of the cutpoints $w_3$, $w_3'$ can appear in $T'$. Moreover we can apply induction hypothesis on  the subgraph  $H'$ obtained removing $v_3'$ from $H$.

By similar argument to the previous case we obtain the assertion under the assumption $w_2'\in V'$. In this case $B_2=C_4$, $v_2\in T$ and $G_{\ol{T}}$ contains the following connected components: the isolated vertex $v_2'$, $B_1'$, obtained removing the vertex $v_2=w_1$ from $B_1$, and the connected components of $H_{\ol{T'}}$ with $H=B_3\cup\cdots\cup B_l$ and $T'=T\setminus \{ w_2',v_2\}$.
\end{proof}
\begin{corollary}\label{Cor:pathunmixed}
 Let $G$ be a graph that satisfies the hypothesis of Proposition \ref{lem:pathcutset}. Then $J(G)$ is unmixed.
\end{corollary}
\begin{proof}
By Lemma \ref{lem:unmixwd} it is sufficient to show that $c(T)=|T|+1$ for all $T\in \CC(G)$. Thanks to Lemma \ref{Lem:pathgraph}, $T=\{w_{i_1}, w_{i_2},\ldots w_{i_t}\}\subseteq \{w_1,\ldots,w_{l-1}\}$ with $1\leq i_1<i_2<\cdots<i_t\leq l-1$ is a cutset of $G$. We note that 
\begin{equation}\label{equT}
 G_{\ol{T}}=G_1\sqcup G_2\sqcup \cdots \sqcup G_{t+1}
\end{equation}
where $G_1=B_1\cup \cdots \cup B_{i_1-1}\cup B''_{i_1}$, $G_{t+1}=B_{i_t+1}'\cup B_{i_t+2}\cup \cdots \cup B_l$ and
\[
G_j=B_{i_j+1}'\cup  B_{i_j+2}\cup\cdots \cup B_{i_{j+1}-1}\cup B''_{i_{j+1}}\mbox{ with }j=2,\ldots,t
\]
and $B_i''$ is the block obtained removing $w_i$ from $B_i$, $B_i'$ is the block obtained removing $w_{i-1}=v_i$ from $B_i$. We easily observe that $G_i$ is connected for $1\leq i\leq t+1$. Hence $c(T)=t+1$.
Now let $T\sqcup T'\in \CC(G)$ with $T=\{w_{i_1}, w_{i_2},\ldots w_{i_t}\}\subseteq \{w_1,\ldots,w_{l-1}\}$ and $T'\cap \{w_1,\ldots, w_{l-1}\}=\emptyset$ with $t'=|T'|>0$. This implies that $v'\in T'$ is either $v_i'$ or $w_i'$ of $B_i\in \{C_4, D\}$. If we focus on $T$  the representation of $G_{\ol{T}}$ in  $\eqref{equT}$ holds in this case, too. Without loss of generality let $v'\in V(G_1)$.  Since the only cutpoint that induces the connected component $G_1$ is $w_{i_1}$, by condition (1) of Proposition \ref{lem:pathcutset}, $v'=v'_{i_1}$. Hence
\begin{equation}
 G_{\ol{T\cup\{v'_{i_1}\}}}=(G_1)_{\ol{\{v'_{i_1}\}}}\sqcup G_2\sqcup \cdots \sqcup G_{t+1}
\end{equation}
where $(G_1)_{\ol{\{v'_{i_1}\}}}=B_1\cup \cdots \cup B_{i_1-1}\sqcup \{v\}$ where $v=f_{i_1}$ if $B_{i_1}=D$ or $v=w'_{i_1}$ if $B_{i_1}=C_4$ and $c(T\cup\{v'_{i_1}\})=(t+1)+1$ where the last summand is induced by the isolated vertex $v$. The same argument holds for all $v'\in T'$.
\end{proof}

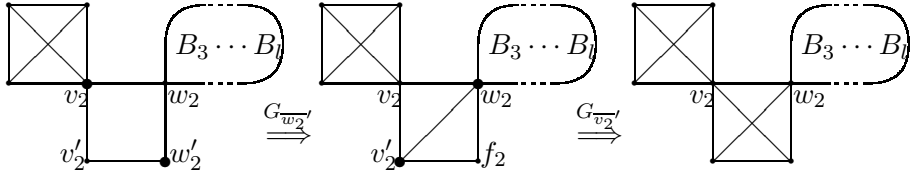
\begin{figure}
\begin{center}
\setlength{\unitlength}{.26cm}
\begin{picture}(45,8)

\newsavebox{\blocksCM}
\savebox{\blocksCM}
  (12,8)[bl]{

  \put(00,04){\circle*{.3}}
  \put(04,04){\circle*{.3}}
  \put(00,08){\circle*{.3}}
  \put(04,08){\circle*{.3}}

  \put(00,04){\line(1,0){4}}
  \put(00,08){\line(1,0){4}}

  \put(00,04){\line(0,1){4}}
  \put(04,04){\line(0,1){4}}

  \put(00,04){\line(1,1){4}}
  \put(04,04){\line(-1,1){4}}

  \put(04,00){\circle*{.3}}
  \put(08,00){\circle*{.3}}
  \put(08,04){\circle*{.3}}

  \put(04,04){\line(1,0){4}}
  \put(04,00){\line(1,0){4}}
  \put(04,04){\line(0,-1){4}}
  \put(08,04){\line(0,-1){4}}

  \put(02.8,03){$v_2$}
  \put(08.1,03){$w_2$}
  
  \put(08,04){\line(1,0){2}}
  \multiput(10.2,04)(0.5,0){4}{\line(1,0){.2}}
  \put(08,04){\line(0,1){2}}
  \multiput(10.2,08)(0.5,0){4}{\line(1,0){.2}}
  
  \put(08.5,05.5){$B_3\cdots B_l$ }

  \qbezier(08,06)(08,08)(10,08)

  \qbezier(12,08)(14,08)(14,06)
  \qbezier(14,06)(14,04)(12,04)
    
}

\put(00,0){\usebox{\blocksCM}}
\put(02.6,00){$v_2'$}
\put(08.2,00){$w_2'$}

\put(04,04){\circle*{.5}}
\put(08,00){\circle*{.5}}

\put(13,01){$\stackrel{G_{\ol{w_2}'}}{\Longrightarrow}$}

\put(16,0){\usebox{\blocksCM}}
\put(20,00){\line(1,1){4}}

\put(18.5,00){$v_2'$}
\put(24.1,00){$f_2$}

\put(20,00){\circle*{.5}}
\put(24,04){\circle*{.5}}

\put(29,01){$\stackrel{G_{\ol{v_2}'}}{\Longrightarrow}$}

\put(32,0){\usebox{\blocksCM}}
\put(36,00){\line(1,1){4}}
\put(40,00){\line(-1,1){4}}

\end{picture}
 
\end{center}
\caption{Cases $B_2 \in \{C_4, D, K_4\}$ of Theorem \ref{lem:pathcm}}\label{CM-C4toK4} 
\end{figure}

\begin{theorem}\label{lem:pathcm}
 Let $G$ be a graph that satisfies the hypothesis of Proposition \ref{lem:pathcutset}. Then $J(G)$ is Cohen-Macaulay.
\end{theorem}
\begin{proof}
We start observing that $\dim S/J(G)=n+1$. This follows by Corollary \ref{Cor:pathunmixed} and the formula (see \cite{HH})
\begin{equation}\label{dimformula}
 \dim S/J(G)=\max\{n-|T|+c(T)\}.
\end{equation}
Hence it is sufficient to prove  that $\depth S/J(G)\geq n+1$ using induction on $l$, the number of blocks of $G$. Our strategy is to focus on the block $B_2$. In particular when $B_2\in \{D,C_4\}$ (see Figure \ref{CM-C4toK4}) we consider the  vertex $v'\in\{v_2',w_2'\}$ and the following exact sequence
\begin{equation}\label{seq1}
0\longrightarrow S/J(G)\longrightarrow S/Q_{v'}\oplus S/J(G')\longrightarrow S/(Q_{v'}+J(G'))\longrightarrow 0
\end{equation}
with $G'=G_{\ol{v'}}$. By Corollary \ref{Lem:pathgraph}, $J(G')$ is a binomial edge ideal and the graph $G'$ satisfies the hypothesis of our Theorem but with a second block that has less cutsets than $B_2$ (see Figure \eqref{CM-C4toK4}). Moreover $S/Q_{v'}$ (respectively $S/Q_{v'}+J(G')$) is a tensor product of  $3$ (respectively $2$) quotient rings whose definining ideals are  binomial edge ideals, and it is Cohen-Macaulay using induction. To obtain this goal we have three cases to study: $B_2=K_{m_2}$, $B_2=D$ and $B_2=C_4$. We start induction with $l=3$. 

Case $B_2=K_{m_2}$.  $S/J(G)$ is Cohen-Macaulay by Theorem 1.1 of \cite{HEH}. 

Case $B_2=D$. We set $G'=G_{\ol{v_2}'}$. By Proposition \ref{Pro:cutset} we have  $E(G')=E(G)\cup\{\{v_2,f_2\}\}$, that is the second block in $G'$ is the complete graph on the vertices $\{v_2,v_2',w_2,f_2\}$. By Corollary \ref{Lem:cutset}, $J(G)=J(G')\cap Q_{v_2'}$, where $Q_{v_2'}=P_{\{v_2',w_2\}}=(x_{v_2'},y_{v_2'})+(x_{w_2},y_{w_2})+J(B_1)+J(B_3')$ and $B_3'$ is obtained by removing the vertex $w_2=v_3$ from the complete graph $B_3$. By Depth Lemma applied on the sequence \eqref{seq1} with $v'=v_2'$ we obtain the assertion. In fact $G'$ is Cohen-Macaulay by the case $B_2=K_{m_2}$ and its $\depth$ is $n+1$. Let $S'=S/(x_{v_2'},y_{v_2'},x_{w_2},y_{w_2})$, we obtain 
\begin{multline}\label{tensor}
 S/Q_{v_2'}\cong S'/J(B_1)+J(B_3')\cong  \\
\cong\frac{K[\{x_i,y_i\}:i\in V(B_1)]}{J(B_1)}\otimes K[x_{f_2},y_{f_2}]\otimes \frac{K[\{x_i,y_i\}:i\in V(B_3')]}{J(B'_3)}
\end{multline}
where the $3$ quotient rings associated to the complete graphs $B_1$, $B_3'$ and the isolated vertex $f_2$  are Cohen-Macaulay. Using the formula \eqref{dimformula} for each ring, and adding the results thanks to \eqref{tensor}, we obtain
\[
 \depth S/Q_{v_2'}=|(V(B_1)+1) +(1+1)+(|V(B_3')|+1)=n+1.
\]
By the same argument $\depth S/(Q_{v_2'}+J(G'))=n$. In fact
\begin{multline}\label{tensorwithdepthn}
 S/Q_{v_2'}+J(G')\cong S'/J(B_1\cup E')+J(B_3')\cong  \\
\cong\frac{K[\{x_i,y_i\}:i\in V(B_1)\cup\{f_2\}]}{J(B_1\cup E)}\otimes \frac{K[\{x_i,y_i\}:i\in V(B_3')]}{J(B'_3)}
\end{multline}
where $E$ is the edge $\{w_1=v_2,f_2\}$. Hence  $J(B_1\cup E)$ is Cohen-Macaulay by Theorem 1.1 of \cite{HEH}. By Depth Lemma the assertion follows.

Case $B_2=C_4$. We set $v'=w_2'$ and $G'=G_{\ol{w_2}'}$. By Proposition \ref{Pro:cutset} we have  $E(G')=E(G)\cup\{\{v_2',w_2\}\}$, that is the second block in $G'$ is a diamond, hence $J(G')$ is Cohen-Macaualy by the case $B_2=D$. Moreover  $\depth S/Q_{v'}=n+1$  and $\depth S/Q_{v'}+J(G')=n$. In fact $Q_{v'}=P_{\{v_2,w_2'\}}=(x_{w_2'},y_{w_2'})+(x_{v_2},y_{v_2})+J(B_1')+J(B_3)$, where $B_1'$ is the complete graph obtained removing the vertex $v_2=w_1$ from the graph $B_1$. Using a representation of $S/Q_{v'}$ similar to \eqref{tensor} we easily obtain the assertion. Moreover $Q_{v'}+J(G')=(x_{w_2'},y_{w_2'})+(x_{v_2},y_{v_2})+J(B_1')+J(B_3\cup E)$ where $E$ is the edge $\{v_2',w_2\}$. Also in this case we obtain the assertion using a representation of $S/Q_{v'}+J(G')$ equivalent to the one used in \eqref{tensorwithdepthn}. 

Let $l>3$. Case $B_2=K_{m_2}$. By Lemma \ref{The:free} applied on the graphs $G_1=B_1$ and $G_2=B_2\cup \cdots \cup B_l$, since $G_1$ is a complete graph and $G_2$ is Cohen-Macaulay by induction hypothesis, we obtain that $G$ is Cohen-Macaulay.

Case $B_2=D$. As in the case $B_2=D$ with $l=3$, setting $G'=G_{\ol{v_2}'}$, the second block in $G'$, namely $B_2'$, is the complete graph on the vertices $\{v_2,v_2',w_2,w_2'\}$, hence
\[
 G'=B_1\cup B_2'\cup B_3\cup \cdots \cup B_l.
\]
Therefore $J(G')$ is Cohen-Macaulay using induction hypothesis and Lemma \ref{The:free} applied on $B_1$ and $B_2'\cup B_3\cup \cdots \cup B_l$.
By Corollary \ref{Lem:cutset} $Q_{v_2'} = \bigcap_{T\in\CC(G), v_2'\in T} P_T(G)$. Thanks to condition (1) of Proposition \ref{lem:pathcutset} for each $T\in \CC(G)$ with $v_2'\in T$ then $w_2\in T$ and $v_2\notin T$. Hence $Q_{v_2'}=J(B_1) + (x_{v_2'},y_{v_2'},x_{w_2},y_{w_2})+J$. We claim that $J=J(H)$, where $H=B_3'\cup B_4\cup \cdots \cup B_l$, with $B_3'$ obtained removing the vertex $w_2=v_3$ from $B_3$. By the claim, and defining $S'=S/(x_{v_2'},y_{v_2'},x_{w_2},y_{w_2})$, we obtain a representation of $S/Q_{v_2'}$ by tensor product similar to \eqref{tensor}, that is
\[
\frac{K[\{x_i,y_i\}:i\in V(B_1)]}{J(B_1)}\otimes K[x_{f_2},y_{f_2}]\otimes \frac{K[\{x_i,y_i\}:i\in V(H)]}{J(H)},
\]
and using induction hypothesis on $H$ we obtain the assertion. The claim follows proving the condition
\begin{equation}\label{cHcG}
T\in \CC(H) \mbox{ if and only if }T\cup \{v_2',w_2\}\in \CC(G). 
\end{equation}
\textbf{Proof of \eqref{cHcG}}. Case $B_3=D$.  Let $B_3'=K_3$ on the vertices $\{v_3',f_3,w_3\}$, then $H$ satisfies Proposition \ref{lem:pathcutset}. Let  $T\in \CC(H)$. To prove that $T\cup \{v_2',w_2\}\in \CC(G)$ it is enough to check conditions (1) and (2) for all the vertices $v'\in T'=(T\cup \{v_2',w_2\})\setminus \{w_1,\ldots, w_{l-1}\}$.
If $v'\in V(B_2)\cap T'$ then $v'=v_2'$ and it satisfies condition (1). If $v'\in V(B_i)\cap T'$ with $4\leq i\leq l$ then it satisfies condition 1. and 2. with respect to $G$ since it satisfies the same conditions with respect to $H$. We end this implication observing that $V(B_3)\cap T'=\emptyset$. If $T\subseteq V(H)$ with $T\notin \CC(H)$ either there exists a block in $\{B_4,\ldots, B_l\}$ such that either (1) or (2) is not satisfied or there exists $v\in V(B_3')$ with $v\neq w_3$. In both cases $T\cup\{v_2',w_2\}$ does not satisfy (1) and (2) with respect to $G$. 

The cases $B_3\in\{K_{m_3},C_4\}$ follow by similar argument. We only point out some facts when $B_3=C_4$. $B_3'$ is a path defined on the vertex set  $\{v_3',w_3',w_3\}$ and edges $\{v_3',w_3'\}$, $\{w_3',w_3\}$. In this case is useful to consider $H=E\cup H'$ where $E$ is the edge $\{v_3',w_3'\}$ and $H'$ is obtained removing the vertex $v_3'$ from $H$. Then $H'$ satisfies the hypothesis of the Proposition and it is Cohen-Macaulay by induction. By Lemma \ref{The:free} applied to $H'$ and $E$  we obtain that $J(H)$ is Cohen-Macaulay, too. Moreover by Lemma 2.3 of \cite{RR} we have
\[
 \CC(H)=\CC(H')\cup \{T\cup\{w_3'\}:T\in \CC(H')\mbox{ with }w_3\notin T\}.
\]
Observe that the graph studied in this case is decomposable with respect to the vertex $v_2$. Hence also the graphs
\begin{equation}\label{eq:alsoDisCM}
 B_2\cup \cdots \cup B_l\mbox{ with }B_2=D\mbox{ is Cohen-Macaulay.} 
\end{equation}
Case $B_2=C_4$. Let $G'=G_{\ol{w_2}'}$.  As in the case $B_2=C_4$ and $l=3$, the second block of $G'$ is $D$. Then $J(G')$ is Cohen-Macaulay by the case $B_2=D$ with $l>3$, and  $J(G)=J(G')\cap Q_{w_2'}$ with $Q_{w_2'} = \bigcap_{T\in\CC(G), w_2'\in T} P_T(G)$. We claim that $Q_{w_2'}=J(B_1')+(x_{w_2'},y_{w_2'})+(x_{v_2},y_{v_2})+J(H)$ where $B_1'$ is the complete graph obtained removing $v_2=w_1$ from $B_1$. Moreover, using the notation introduced in Proposition \ref{Pro:cutset}, 
\[
H=(B_3\cup B_4\cup\cdots\cup B_l)_{\ol{v_3}}=(B_3)_{\ol{v_3}}\cup B_4\cup\cdots\cup B_l.
\]
By the claim it follows, defining $S'=S/(x_{v_2},y_{v_2},x_{w_2'},y_{w_2'})$, the following representation of $S/Q_{w_2'}$
\[
\frac{K[\{x_i,y_i\}:i\in V(B_1')]}{J(B_1')}\otimes K[x_{v_2'},y_{v_2'}]\otimes \frac{K[\{x_i,y_i\}:i\in V(H)]}{J(H)}.
\]
We focus on the last factor since the other ones are exactly equivalent to the ones already studied. The block $(B_3)_{\ol{v_3}}$ is either a complete graph or a diamond graph with $v_3$ a free vertex. In both cases it is Cohen-Macaulay by induction hypothesis and \eqref{eq:alsoDisCM}. The claim follows by the condition
\begin{equation}\label{cHcG1}
T\in \CC(H) \mbox{ if and only if }T\cup \{v_2,w_2'\}\in \CC(G). 
\end{equation}
To prove \eqref{cHcG1} we use Proposition \ref{lem:pathcutset} and similar arguments to the ones used to prove \eqref{cHcG}.
\end{proof}

\begin{lemma}\label{lem:indecomposablepathcm}
Let $G$ be an indecomposable cactus graph such that $B(G)$ is a path.  We use notation \eqref{pathnotation}. The following conditions are equivalent:
\begin{enumerate}
 \item $J(G)$ is Cohen-Macaulay;
 \item $J(G)$ is unmixed;
 \item One of the following $2$ cases occurs
 \begin{enumerate}
  \item $G\in \{K_2,C_3\}$.
  \item $G$ has $C_4$ subgraphs that satisfy the $(C4)$-condition, $l\geq 3$ and 
  \begin{enumerate}
   \item $B_1,B_l\in \{C_3,K_2\}$,
   \item $B_2$ and $B_{l-1}$ are $C_4$,
   \item $B_i\in \{C_3,C_4\}$ for $3\leq i\leq l-2$ and if $B_i=C_3$ then $B_{i+1}=C_4$.
  \end{enumerate} 
 \end{enumerate}
\end{enumerate}
\end{lemma}
\begin{proof} $(1)\Rightarrow (2)$. It is a known fact. $(2)\Rightarrow (3)$. By Proposition \ref{the:C3AndC4Only} only $K_2$, $C_3$ and $C_4$ are admissible blocks. Since every vertex of a complete graph is a free vertex, we observe that every cactus graph whose blocks are $K_2$ and $C_3$ and not $C_4$ is Cohen-Macaulay and is decomposable in single blocks. Now suppose that the block graph is a path containing one or more blocks that are $C_4$. By Proposition \ref{Pro:C4} a block $C_4$ has two cutpoints, thus neither $B_1$ nor $B_l$ is $C_4$. Since by hypothesis $G$ is indecomposable, two complete graphs cannot be adjacent in $B(G)$. We end observing that if a $K_2$ is between two cycles $C_4$ it is not unmixed (see also Remark 4.7 of \cite{BMS}). 
 $(3)\Rightarrow (1)$. The implication follows applying Theorem \ref{lem:pathcm}. 
\end{proof}

\begin{example}
Thanks to Lemma \ref{lem:indecomposablepathcm} we obtain the Cohen-Macaulay indecomposable graph of trees and unicyclic graphs  shown in Figure \ref{CM-TreeUni}. The bicyclic with the same properties are in Figure \ref{CM-cactus}. From now on we underline the free vertices by a circle around them.

\begin{figure}
\begin{center}
\setlength{\unitlength}{.3cm}
\begin{picture}(24,8)

\put(00,04){\circle*{.2}}
\put(00,04){\circle{.5}}

\put(04,04){\circle*{.2}}
\put(04,04){\circle{.5}}

\put(00,04){\line(1,0){4}}

\put(08,04){\circle*{.2}}
\put(10,07){\circle*{.2}}
\put(12,04){\circle*{.2}}
\put(08,04){\circle{.5}}
\put(10,07){\circle{.5}}
\put(12,04){\circle{.5}}

\put(08,04){\line(2,3){2}}
\put(08,04){\line(1,0){4}}
\put(10,07){\line(2,-3){2}}

\put(16,04){\circle*{.2}}
\put(14,07){\circle*{.2}}
\put(14,07){\circle{.5}}

\put(16,04){\line(-2,3){2}}

\put(16,00){\circle*{.2}}
\put(20,00){\circle*{.2}}
\put(20,04){\circle*{.2}}

\put(16,04){\line(1,0){4}}
\put(16,00){\line(1,0){4}}
\put(16,04){\line(0,-1){4}}
\put(20,04){\line(0,-1){4}}

\put(20,04){\line(2,3){2}}
\put(22,07){\circle*{.2}}
\put(22,07){\circle{.5}}

\end{picture}
 
\end{center}
\caption{Tree and unicyclic Cohen-Macaulay indecomposable graphs.}\label{CM-TreeUni} 
\end{figure}
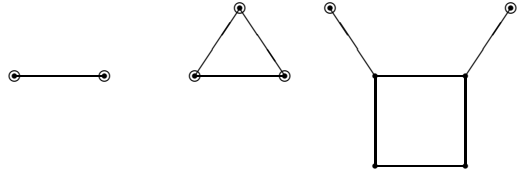

\begin{figure}
\begin{center}
\setlength{\unitlength}{.3cm}
\begin{picture}(24,8)

\put(00,04){\circle*{.2}}
\put(00,04){\circle{.5}}

\put(04,04){\circle*{.2}}

\put(00,04){\line(1,0){4}}

\put(04,00){\circle*{.2}}
\put(08,00){\circle*{.2}}
\put(08,04){\circle*{.2}}

\put(04,04){\line(1,0){4}}
\put(04,00){\line(1,0){4}}
\put(04,04){\line(0,-1){4}}
\put(08,04){\line(0,-1){4}}

\put(10,07){\circle*{.2}}
\put(12,04){\circle*{.2}}
\put(10,07){\circle{.5}}
\put(12,04){\circle{.5}}

\put(08,04){\line(2,3){2}}
\put(08,04){\line(1,0){4}}
\put(10,07){\line(2,-3){2}}

\put(16,04){\circle*{.2}}
\put(14,07){\circle*{.2}}
\put(14,07){\circle{.5}}

\put(16,04){\line(-2,3){2}}

\put(16,00){\circle*{.2}}
\put(20,00){\circle*{.2}}
\put(20,04){\circle*{.2}}

\put(16,04){\line(1,0){4}}
\put(16,00){\line(1,0){4}}
\put(16,04){\line(0,-1){4}}
\put(20,04){\line(0,-1){4}}

\put(20,08){\circle*{.2}}
\put(24,08){\circle*{.2}}
\put(24,04){\circle*{.2}}

\put(20,08){\line(1,0){4}}
\put(20,04){\line(1,0){4}}
\put(20,08){\line(0,-1){4}}
\put(24,08){\line(0,-1){4}}

\put(26,01){\circle*{.2}}
\put(26,01){\circle{.5}}

\put(24,04){\line(2,-3){2}}

\end{picture}
 
\end{center}
\caption{Bicyclic Cohen-Macaulay indecomposable cactus graphs.}\label{CM-cactus} 
\end{figure}
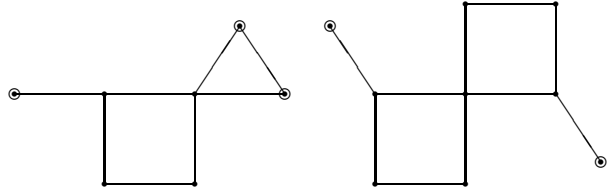

\end{example}

\begin{theorem}\label{Th:DecompositionByUndecomposablePath}
 Let $G$ be a cactus graph.  We use notation \eqref{pathnotation}. The following conditions are equivalent
\begin{enumerate}
 \item $J(G)$ is Cohen-Macaulay;
 \item $J(G)$ is unmixed;
 \item $B(G)$ is a tree, $G$ is decomposable into indecomposable graphs $G_1,\ldots, G_r$, and such that $B(G_i)$ is a path and $G_i$ satisfies one of the equivalent conditions of Lemma \ref{lem:indecomposablepathcm} for  $1\leq i\leq r$.
\end{enumerate}
\end{theorem}
\begin{proof}
 $(1)\Rightarrow (2)$. It is a known fact. $(3)\Rightarrow (1)$. It follows by  Lemma \ref{The:free}. 
 $(2)\Rightarrow (3)$. By Proposition \ref{tree} the block graph of $G$ is a tree. Let $t$ be the number of vertices of $B(G)$ having degree greater than two. We make induction on $t$. If $t=0$ then the block graph $B(G)$ is a path and the assertion follows by Lemma \ref{The:free}  and Lemma \ref{lem:indecomposablepathcm}.
 Let $t>0$ and let $B$ be a a vertex  of  $B(G)$ whose degree is greater than $2$. By Proposition \ref{the:C3AndC4Only} and Proposition \ref{Pro:C4}, $B=C_3$. Suppose that $C_3$ has a vertex $v$ such that $\{v\}=V(G_1)\cap V(G_2)$,  $G=G_1\cup G_2$ and $v$ is a free vertex of $\Delta(G_1)$ and $\Delta(G_2)$. Since $G_1$ and $G_2$ are cactus with number of blocks of degree greater than $2$  less than $t$, by induction hypothesis we are done. Suppose by contradiction that in each of the $3$ vertices of $C_3$ the graph $G$ is not decomposable. Then $C_3$ is adjacent to $3$ blocks $C_4$ as in figure \ref{Non-unmixed-cactus}, where $G_1$, $G_2$ and $G_3$ are intended as cactus subgraphs of $G$. We observe that the set $T$ containing the $6$ vertices, indicated as filled dots in figure, is a cutset. But $G_{\ol{T}}$ has $6$ components. Hence $J(G)$ is not unmixed. Contradiction.   
\begin{figure}
\begin{center}
\setlength{\unitlength}{.3cm}
\begin{picture}(20,12)
 
\newsavebox{\subgraph}
\savebox{\subgraph}(12,8)[bl]{
  \qbezier(01,00)(00,00)(00,01)
  \put(01,00){\line(1,0){2}}
  \qbezier(03,00)(04,00)(04,01)
  \put(04,01){\line(0,1){2}}
  \qbezier(04,03)(04,04)(03,04)
  \put(01,04){\line(1,0){2}}
  \qbezier(01,04)(00,04)(00,03)
  \put(00,01){\line(0,1){2}}
}

\put(00,02){\usebox{\subgraph}}
\put(01.5,03.6){$G_1$}

\put(16,02){\usebox{\subgraph}}
\put(17.5,03.6){$G_2$}

\put(06,09){\usebox{\subgraph}}
\put(7.5,10.6){$G_3$}


\put(04,04){\circle*{.2}}


\put(04,00){\circle*{.5}}
\put(08,00){\circle*{.2}}
\put(08,04){\circle*{.5}}

\put(04,04){\line(1,0){4}}
\put(04,00){\line(1,0){4}}
\put(04,04){\line(0,-1){4}}
\put(08,04){\line(0,-1){4}}

\put(10,07){\circle*{.5}}
\put(12,04){\circle*{.5}}

\put(08,04){\line(2,3){2}}
\put(08,04){\line(1,0){4}}
\put(10,07){\line(2,-3){2}}

\put(12,00){\circle*{.2}}
\put(16,00){\circle*{.5}}
\put(16,04){\circle*{.2}}

\put(12,04){\line(1,0){4}}
\put(12,00){\line(1,0){4}}
\put(12,04){\line(0,-1){4}}
\put(16,04){\line(0,-1){4}}


\put(10,11){\circle*{.2}}
\put(14,07){\circle*{.2}}
\put(14,11){\circle*{.5}}

\put(10,07){\line(1,0){4}}
\put(10,11){\line(1,0){4}}
\put(10,11){\line(0,-1){4}}
\put(14,11){\line(0,-1){4}}

\end{picture}
\end{center}
\caption{A non unmixed cactus graphs with $B(G)$ that is a tree.}\label{Non-unmixed-cactus}  
\end{figure}
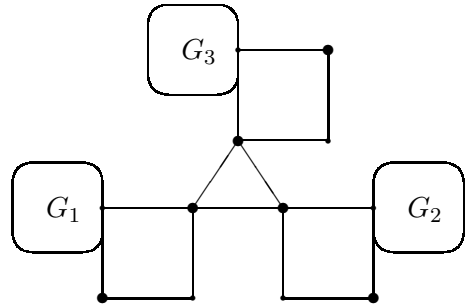
\end{proof}
\begin{example}
In figure \ref{CM-cactus-decomposable} we have that $G$ is the union of $4$ indecomposable Cohen-Macaulay cactus graphs joined by free vertices (surrounded by a circle).  That are : a $K_2$; a $C_3$; one containing a $K_2$, a $C_3$ and a $C_4$; one containing $2$ $K_2$, one $C_3$ and  $3$ $C_4$.  
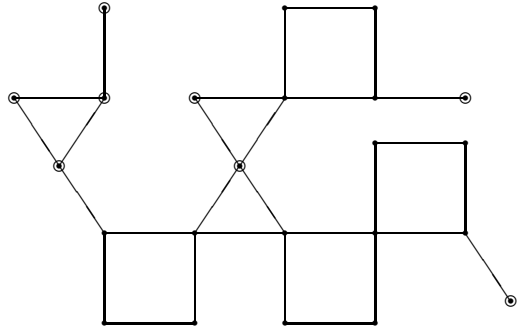
\begin{figure}
\begin{center}
\setlength{\unitlength}{.3cm}
\begin{picture}(24,14)

\put(04,04){\circle*{.2}}
\put(02,07){\circle{.5}}
\put(02,07){\circle*{.2}}


\put(02,07){\line(2,-3){2}}

\put(00,10){\circle*{.2}}
\put(00,10){\circle{.5}}
\put(04,10){\circle*{.2}}
\put(04,10){\circle{.5}}

\put(00,10){\line(2,-3){2}}
\put(00,10){\line(1,0){4}}
\put(02,07){\line(2,3){2}}


\put(04,14){\circle*{.2}}
\put(04,14){\circle{.5}}
\put(04,10){\line(0,1){4}}

\put(04,00){\circle*{.2}}
\put(08,00){\circle*{.2}}
\put(08,04){\circle*{.2}}

\put(04,04){\line(1,0){4}}
\put(04,00){\line(1,0){4}}
\put(04,04){\line(0,-1){4}}
\put(08,04){\line(0,-1){4}}

\put(10,07){\circle*{.2}}
\put(12,04){\circle*{.2}}
\put(10,07){\circle{.5}}

\put(08,04){\line(2,3){2}}
\put(08,04){\line(1,0){4}}
\put(10,07){\line(2,-3){2}}

\put(08,10){\circle*{.2}}
\put(08,10){\circle{.5}}

\put(12,10){\circle*{.2}}

\put(08,10){\line(2,-3){2}}
\put(08,10){\line(1,0){4}}
\put(10,07){\line(2,3){2}}

\put(12,14){\circle*{.2}}
\put(16,14){\circle*{.2}}
\put(16,10){\circle*{.2}}

\put(12,10){\line(1,0){4}}
\put(12,14){\line(1,0){4}}
\put(12,10){\line(0,1){4}}
\put(16,14){\line(0,-1){4}}

\put(20,10){\circle*{.2}}
\put(20,10){\circle{.5}}

\put(16,10){\line(1,0){4}}

\put(12,00){\circle*{.2}}
\put(16,00){\circle*{.2}}
\put(16,04){\circle*{.2}}

\put(12,04){\line(1,0){4}}
\put(12,00){\line(1,0){4}}
\put(12,04){\line(0,-1){4}}
\put(16,04){\line(0,-1){4}}

\put(16,08){\circle*{.2}}
\put(20,08){\circle*{.2}}
\put(20,04){\circle*{.2}}

\put(16,08){\line(1,0){4}}
\put(16,04){\line(1,0){4}}
\put(16,08){\line(0,-1){4}}
\put(20,08){\line(0,-1){4}}

\put(22,01){\circle*{.2}}
\put(22,01){\circle{.5}}

\put(20,04){\line(2,-3){2}}
\end{picture}
\end{center}
\caption{A Cohen-Macaulay cactus graphs.}\label{CM-cactus-decomposable} 
\end{figure}
\end{example}
\section{Classification of Cohen--Macaulay bicyclic graphs}\label{sec:Bicycle}
In this section thanks to the classification of Cohen-Macaulay binomial edge ideals of cactus graphs we classify the ones that are bicyclic, namely the ideals of deviation $2$.
In fact a cactus graph having $2$ cycles as blocks is in particular a bicyclic graph. Hence we focus our attention on bicyclic graphs $G$ that are not cactus. In this case there exists one block $B_1$ in $G$ such that
\begin{equation}\label{eq:B1}
 B_1=\bigcup_{i=1}^3 P_i
\end{equation}
where $P_1$, $P_2$ and $P_3$ are paths, $V(P_i)\cap V(P_j)=\{a,b\}$ for  $1\leq i<j\leq 3$ and if $B$ is a block of $G$ with $B\neq B_1$ then $B$ is an edge.
\begin{remark}\label{rem:l123}
 The set $T=\{a,b\}$, where $a$ and $b$ are defined in \eqref{eq:B1}, is a cutset. Assume from now on that the ideal $J(G)$ is unmixed. Let $l_i$ be the length of the path $P_i$ in \eqref{eq:B1} such that
 \[
  1\leq l_1\leq l_2\leq l_3.
 \]
 If $l_1=1$, that is $\{a,b\}\in E(G)$, then $a$ is a cutpoint and $b$ is not a cutpoint. In fact by Lemma \ref{lem:unmixwd}, $c(T)=3$. But we have exactly two connected components induced by the paths $P_2$ and $P_3$. Hence there exists another connected component in $G_{\ol{T}}$ that is not a subgraph of $B_1$ and the assertion follows. By a similar argument if $l_1>1$ then neither $a$ nor $b$ is a cutpoint.
\end{remark}
All over the section we use the notation defined in \eqref{eq:B1} and Remark \ref{rem:l123}. We call these graphs \textit{non-cactus} bicyclic graphs.
\begin{lemma}\label{lem:length} 
 Let $G$ be a non-cactus bicyclic graph such that $J(G)$ is unmixed. We use notation \eqref{eq:B1}. Then
 \begin{enumerate}
  \item each path $P_i$ has length less than $4$;
  \item at most one path  $P_i$ has length $3$. 
 \end{enumerate}
\end{lemma}
\begin{proof}
  $(1)$. Suppose by contradiction that exists a path $P\in \{P_1,P_2,P_3\}$ of length $l\geq 4$. Let $a$, $a_1$, $c$, $b_1$ and $b$, 5 distinct vertices of $V(P)$,
 with $\{a,a_1\},\{b,b_1\}\in E(P)$. Let $T=\{a,b\}$ and suppose that neither $a$ nor $b$ is a cutpoint. We observe that $T_1=\{a,b_1\}$ is a cutset. Hence $c(T_1)=3$. But in $G_{\ol{T}_1}$ there are two connected components that are subgraphs of $B_1$, one containing the vertex $c$ and one containing the vertex $b$.  Hence there exists another block $B_2$ different from $B_1$ that contains the vertex $b_1$. Let $E(B_2)=\{\{b_1,b_1'\}\}$. By the same argument there exists a block $B_3$ such that $E(B_3)=\{\{a_1,a_1'\}\}$. The set $T_2=\{a_1,b_1\}$ is a cutset and induces $4$ components in $G_{\ol{T}_2}$: The first containing $a$ and $b$, the second containing $a_1'$, the third containing $b_1'$, and the fourth containing $c$. Contradiction. By similar argument is left to the reader to check that if $a$ cutpoint and $b$ is not a cutpoint, then $T_1=\{b,c\}$ is a cutset with $c$ a cutpoint. Hence we have a contradiction focusing on the cutset $\{a,c\}$.
 
  $(2)$. Suppose by contradiction that exist $2$ paths $P_1$ and $P_2$ of length $3$ such that
 \[
  V(P_i)=\{a,a_i,b_i,b\},\, E(P_i)=\{\{a,a_i\},\{a_i,b_i\},\{b_i,b\}\}\mbox{ for }i=1,2. 
 \]
  By Remark \ref{rem:l123}  we assume the vertex $b$ is not a cutpoint and we focus on the cutset $T=\{b,a_1\}$. $G_{\ol{T}}$ has two connected components that are subgraphs of $B_1$, one containing the vertex $a$ and one containing the isolated vertex $b_1$.  Hence there exists another block $B_2$ different from $B_1$ that contains the vertex $a_1$. Let $E(B_2)=\{\{a_1,a_1'\}\}$. By the same argument there exists a block $B_3$ such that $E(B_3)=\{\{a_2,a_2'\}\}$. The set $T_1=\{b,a_1,a_2\}$ is a cutset with $5$ components in $G_{\ol{T}_1}$: the first containing $a$, the second (respectively the third)  containing $a_1'$ (respectively $a_2'$)  and the $2$ isolated vertices $b_1$ and $b_2$. Contradiction.
 \end{proof}

\begin{lemma}
Let $G$ be a non-cactus bicyclic graph. Then $J(G)$ is Cohen-Macaulay (respectively unmixed but not Cohen-Macaulay) if and only if $B(G)$ is a tree and $G$ is decomposable into indecomposable graphs $G_1,\ldots, G_r$, and such that $G_1$ is one of the graphs in Figure \ref{CM-noncactus} (respectively Figure \ref{Unmixed-noncactus}).
\end{lemma}
\begin{proof}
 Let $(l_1,l_2,l_3)\in \NN^3$ where  $l_i$ is the length of the paths $P_i$. By Remark \ref{rem:l123} and Lemma \ref{lem:length} we have to study the following $4$ cases:
\[
 (1,2,2),(1,2,3), (2,2,2),(2,2,3). 
\]

\begin{figure}[hbt]
\begin{center}
\setlength{\unitlength}{.3cm}
\begin{picture}(28,9)



\put(04,04){\circle*{.2}}

\put(04,08){\circle*{.2}}
\put(04,08){\circle{.5}}

\put(04,04){\line(0,1){4}}

\put(04,00){\circle*{.2}}
\put(04,00){\circle{.5}}

\put(08,00){\circle*{.2}}
\put(08,04){\circle*{.2}}
\put(08,04){\circle{.5}}

\put(04,04){\line(1,0){4}}
\put(04,00){\line(1,0){4}}
\put(04,04){\line(0,-1){4}}
\put(08,04){\line(0,-1){4}}
\put(04,04){\line(1,-1){4}}

\put(18,07){\circle*{.2}}
\put(18,07){\circle{.5}}

\put(16,04){\circle*{.2}}
\put(20,04){\circle*{.2}}

\put(16,04){\line(2,3){2}}
\put(16,04){\line(1,0){4}}
\put(18,07){\line(2,-3){2}}

\put(16,00){\circle*{.2}}
\put(20,00){\circle*{.2}}
\put(20,04){\circle*{.2}}

\put(24,00){\circle*{.2}}
\put(24,04){\circle*{.2}}

\put(24,00){\circle{.5}}
\put(24,04){\circle{.5}}

\put(16,04){\line(1,0){4}}
\put(16,00){\line(1,0){4}}
\put(16,04){\line(0,-1){4}}
\put(20,04){\line(0,-1){4}}

\put(20,04){\line(1,0){4}}
\put(20,00){\line(1,0){4}}

\put(03,4){$a$}
\put(02.5,0){$f_1$}
\put(02.5,8){$f_3$}

\put(08.2,4){$f_2$}
\put(08.2,0){$b$}

\put(15,4){$b$}
\put(14.5,0){$b_1$}

\put(20.2,4.4){$a$}
\put(20.2,0.4){$a_1$}

\put(24.2,4){$f_2$}
\put(24.2,0){$f_3$}

\put(17.5,07.5){$f_1$}

\end{picture}
\caption{Bicyclic Cohen-Macaulay non-cactus graphs.}\label{CM-noncactus} 
\end{center}

\end{figure}
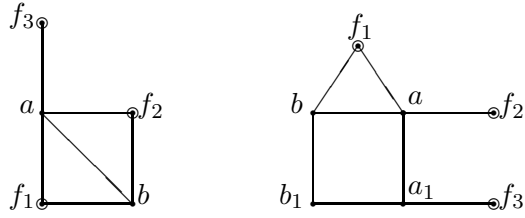

Case $(1,2,2)$. Since $l_1=1$  we obtain the graph on the left in Figure \ref{CM-noncactus} that is Cohen-Macaulay. In fact, if we consider the cone from the vertex $a$ to the two connected components given by the isolated vertex $f_3$ and the path whose edges are $\{f_1,b\}$ and $\{b,f_2\}$, by Theorem 3.8 of \cite{RR}, we have the assertion.

Case $(1,2,3)$. Since $l_1=1$ we obtain the graph on the right in figure \ref{CM-noncactus} removing the vertex  $f_3$. Using similar argument to the one used in Proposition \ref{Pro:C4}, we obtain that the only candidate is exactly the graph in figure \ref{CM-noncactus} and
\[
 \mathcal{C}(G)=\{\emptyset, \{a\},\{a_1\},\{a,a_1\}, \{a,b\}, \{a,b_1\}, \{a_1,b\},\{a,a_1,b\}\}.
\]
We focus on the cutpoint $a$. By Corollary \ref{Lem:cutset} we obtain $J(G)=J(G_{\ol{a}}) \cap Q_a$, moreover $Q_a=(x_a,y_a)+ J(P)$ where $J(P)$ is the binomial edge ideal of the path with
\[
 E(P)=\{\{f_1,b\},\{b,b_1\},\{b_1,a_1\},\{a_1,f_3\}\},  
\] 
\[
 \mathcal{C}(P)=\{\emptyset, \{a_1\}, \{b\}, \{b_1\},\{a_1,b\}\}.
\]
The ring $S/Q_a$ is Cohen-Macaulay of dimension $8$ and it has two components: The path $P$ on $5$ vertices and the isolated vertex $f_2$.  We also observe that $J(G_{\ol{a}})$ is the cone from the vertex $a_1$ to the vertex $f_3$ and the graph obtained attaching the edge $\{b_1, b\}$ to the complete graph whose vertices are $\{a,b,f_1,f_2\}$. Hence $J(G_{\ol{a}})$ is Cohen-Macaulay of dimension $8$, too. Moreover $Q_a+J(G_{\ol{a}})$ is a binomial edge ideal that is equal to the  previous cone removing the vertex $a$, hence it is Cohen-Macaulay of $\depth$ $7$.
By the Depth Lemma applied on the following exact sequence the assertion follows
\[
0\longrightarrow S/J(G)\longrightarrow S/Q_a\oplus S/J(G_{\ol{a}})\longrightarrow S/(Q_a+J(G_{\ol{a}}))\longrightarrow 0.
\]

Cases $(2,2,2)$ and $(2,2,3)$. These two cases are unmixed if and only if we add the edges as in figure \ref{Unmixed-noncactus}. Moreover we already found them in \cite{Ri}, Example 3.2 and 3.3. In that paper by symbolic computation we observed that they are not Cohen-Macaulay. In \cite{BMS} there is an argument for the non Cohen-Macaualyness of the bipartite one, the one on the right of figure \ref{Unmixed-noncactus}.
\begin{figure}[hbt]
\begin{center}
\setlength{\unitlength}{.3cm}
\begin{picture}(28,7)

\put(-1,03.5){$f_1$}
\put(11.5,03.5){$f_2$}

\put(00,03){\circle*{.2}}
\put(00,03){\circle{.5}}

\put(03,03){\circle*{.2}}
\put(06,03){\circle*{.2}}
\put(06,00){\circle*{.2}}
\put(06,06){\circle*{.2}}

\put(09,03){\circle*{.2}}
\put(12,03){\circle*{.2}}

\put(12,03){\circle{.5}}

\put(00,03){\line(1,0){3}}
\put(06,00){\line(0,1){6}}
\put(09,03){\line(1,0){3}}

\put(03,03){\line(1,1){3}}
\put(03,03){\line(1,-1){3}}

\put(09,03){\line(-1,1){3}}
\put(09,03){\line(-1,-1){3}}


\put(15,03.5){$f_1$}

\put(27.5,00.5){$f_2$}

\put(27.5,06.5){$f_3$}

\put(16,03){\circle*{.2}}
\put(16,03){\circle{.5}}

\put(19,03){\circle*{.2}}
\put(22,03){\circle*{.2}}
\put(22,00){\circle*{.2}}
\put(22,06){\circle*{.2}}

\put(25,00){\circle*{.2}}
\put(25,06){\circle*{.2}}

\put(28,00){\circle*{.2}}
\put(28,06){\circle*{.2}}

\put(28,00){\circle{.5}}
\put(28,06){\circle{.5}}

\put(16,03){\line(1,0){3}}
\put(22,00){\line(0,1){6}}
\put(22,00){\line(1,0){6}}
\put(22,06){\line(1,0){6}}

\put(25,00){\line(0,1){6}}

\put(19,03){\line(1,1){3}}
\put(19,03){\line(1,-1){3}}
\end{picture}

\caption{Bicyclic unmixed but non Cohen-Macaulay non-cactus graphs.}\label{Unmixed-noncactus} 
\end{center} 
\end{figure}
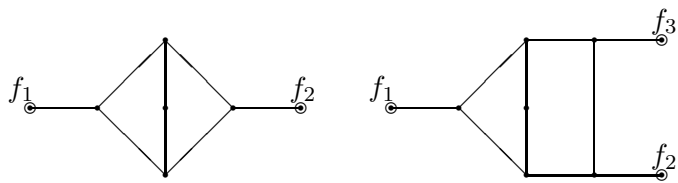

\end{proof}
Now we are ready to give the main result of the section.
\begin{corollary}
 Let $G$ be a bicyclic graph. Then  $J(G)$ is Cohen-Macaulay if and only if $B(G)$ is a tree, $G$ is decomposable into indecomposable graphs $G_1,\ldots, G_r$, and such that one of the following cases occurs:
 \begin{enumerate}
  \item $G_1$ and $G_2$ are in the set of unicyclic graphs in Figure \ref{CM-TreeUni};
  \item $G_1$ is one of the bicyclic cactus graphs in Figure \ref{CM-cactus};
  \item $G_1$ is one of the bicyclic non-cactus graphs in Figure \ref{CM-noncactus}.
 \end{enumerate}
\end{corollary}

\section*{Acknowledgements}
The author was partially supported by GNSAGA of INdAM (Italy).

\end{document}